\input{diagrams.tex}
\diagramstyle[scriptlabels,height=8mm,width=8mm]

\documentclass[a4paper,12pt]{amsart}
\usepackage[utf8]{inputenc}
\usepackage[english]{babel} 

\usepackage[T1]{fontenc}

\usepackage{amsthm}
\usepackage{amsmath}
\usepackage{amsfonts}
\usepackage{mathtools}
\usepackage[colorlinks]{hyperref}
\usepackage{enumerate}
\usepackage{amssymb}
\usepackage{tikz,pgf}
\usepackage{graphicx}
\usetikzlibrary{matrix,arrows,decorations.pathmorphing}
\usetikzlibrary{calc}
\makeatletter
\newbox\dottedarrow@box
\setbox\dottedarrow@box\hbox
  {%
    \begin{tikzpicture}
      \draw[dotted,->] (0,0) -- (1.5em,0);
    \end{tikzpicture}%
  }
\newcommand*\dottedarrow
  {\relax\ifmmode\expandafter\dottedarrow@m\else\expandafter\dottedarrow@t\fi}
\newcommand*\dottedarrow@t[1][1.5em]
  {\resizebox{#1}{!}{\raisebox{.5ex}{\usebox\dottedarrow@box}}}
\newcommand*\dottedarrow@m[1][]
  {%
    \if\relax\detokenize{#1}\relax
      \mathchoice
        {\dottedarrow@t}
        {\dottedarrow@t}
        {\dottedarrow@t[1.1em]}
        {\dottedarrow@t[0.9em]}%
    \else
      \dottedarrow@t[#1]%
    \fi
  }
\makeatother

\def\bdi{\begin{diagram}}
\def\edi{\end{diagram}}

\newtheorem{thm}{Theorem}[section]
\newtheorem{cor}[thm]{Corollary}
\newtheorem{lem}[thm]{Lemma}
\newtheorem{prop}[thm]{Proposition}
\theoremstyle{definition}
\newtheorem{defi}[thm]{Definition}
\newtheorem{defis}[thm]{Definitions}
\newtheorem{conj}[thm]{Conjecture}
\newtheorem{conv}[thm]{Convention}
\newtheorem{nota}[thm]{Notation}
\newtheorem{rem}[thm]{Remark}
\newtheorem{rems}[thm]{Remarks}
\newtheorem{exa}[thm]{Example}
\newtheorem{exas}[thm]{Examples}
\newtheorem{prob}[thm]{Problem}
\newtheorem{probs}[thm]{Problems}
\newtheorem{ques}[thm]{Question}
\newtheorem{sett}[thm]{Setting}
\newtheorem{sit}[thm]{}

\newcommand{\Spec}{\operatorname{{\rm Spec}}}

\def\codim{\mathop{\rm codim}}

\def\Pic{\mathop{\rm Pic}}

\renewcommand{\epsilon}{\varepsilon}

\def\and{\quad\mbox{and}\quad}

\newcommand{\G}{\ensuremath{\mathbb{G}}}
\newcommand{\GG}{\ensuremath{\mathbb{G}}}

\newcommand{\A}{\ensuremath{\mathbb{A}}}

\newcommand{\kk}[1]{\bk^{[#1]}}

\newcommand{\hX}{{\hat X}}

\newcommand{\hE}{{\hat E}}

\newcommand{\hO}{{\hat O}}

\newcommand{\hp}{{\hat p}}
\newcommand{\hs}{{\hat s}}

\newcommand{\tF}{{\tilde F}}

\newcommand{\tX}{{\tilde X}}
\newcommand{\tY}{{\tilde Y}}

\newcommand{\tE}{{\tilde E}}

\newcommand{\tp}{{\tilde p}}
\newcommand{\ts}{{\tilde s}}

\newcommand{\cL}{{\ensuremath{\mathcal{L}}}}

\newcommand{\cA}{{\ensuremath{\mathcal{A}}}}

\newcommand{\cO}{{\ensuremath{\mathcal{O}}}}

\renewcommand{\rho}{\varrho}

\def\bals#1\eals{\begin{align*}#1\end{align*}}
\def\bal#1\eal{\begin{align}#1\end{align}}

\def\kk{{\mathbb K}}
\def\A{{\mathbb A}}

\def\ZZ{{\mathbb Z}}
\def\QQ{{\mathbb Q}}

\def\CC{{\mathbb C}}

\def\PP{{\mathbb P}}

\renewcommand{\phi}{\varphi}

\newcommand{\bnum}{\begin{enumerate}}
\newcommand{\enum}{\end{enumerate}}

\newcommand{\brem}{\begin{rem}}
\newcommand{\brems}{\begin{rems}}
\newcommand{\erem}{\end{rem}}
\newcommand{\erems}{\end{rems}}
\newcommand{\bprob}{\begin{prob}}
\newcommand{\eprob}{\end{prob}}
\newcommand{\bprobs}{\begin{probs}}
\newcommand{\eprobs}{\end{probs}}
\newcommand{\bques}{\begin{ques}}
\newcommand{\eques}{\end{ques}}
\newcommand{\bexa}{\begin{exa}}
\newcommand{\bexas}{\begin{exas}}
\newcommand{\eexa}{\end{exa}}
\newcommand{\eexas}{\end{exas}}
\newcommand{\bdefi}{\begin{defi}}
\newcommand{\edefi}{\end{defi}}
\newcommand{\bdefis}{\begin{defis}}
\newcommand{\edefis}{\end{defis}}
\newcommand{\bcor}{\begin{cor}}
\newcommand{\ecor}{\end{cor}}
\newcommand{\blem}{\begin{lem}}
\newcommand{\elem}{\end{lem}}
\newcommand{\bconv}{\begin{conv}}
\newcommand{\econv}{\end{conv}}
\newcommand{\bconj}{\begin{conj}}
\newcommand{\econj}{\end{conj}}
\newcommand{\bprop}{\begin{prop}}
\newcommand{\eprop}{\end{prop}}
\newcommand{\bthm}{\begin{thm}}
\newcommand{\ethm}{\end{thm}}
\newcommand{\bnota}{\begin{nota}}
\newcommand{\enota}{\end{nota}}
\newcommand{\bsit}{\begin{sit}}
\newcommand{\esit}{\end{sit}}
\newcommand{\be}{\begin{equation}}
\newcommand{\ee}{\end{equation}}
\newcommand{\bproof}{\begin{proof}}
\newcommand{\eproof}{\end{proof}}
\newcommand{\bsett}{\begin{sett}}
\newcommand{\esett}{\end{sett}}
\def\ba{\begin{array}}
\def\ea{\end{array}}

\begin{document}
\title[Varieties covered by affine spaces]{
Varieties covered by affine spaces, 
uniformly rational varieties and their cones}

\author{I.~Arzhantsev, S.~Kaliman and M.~Zaidenberg}

\address{HSE University, 
Faculty of Computer Science,
Pokrovsky Boulevard 11, Moscow, 109028 Russia}
\email{arjantsev@hse.ru}
\address{Department of Mathematics,
University of Miami, Coral Gables, FL 33124, USA}
\email{kaliman@math.miami.edu}
\address{Univ. Grenoble Alpes, CNRS, IF, 38000 Grenoble, France}
\email{mikhail.zaidenberg@univ-grenoble-alpes.fr}

\dedicatory{In memory of Dmitri N. Akhiezer}

\thanks{The first author was supported 
by the grant RSF-DST 22-41-02019}

\date{}
\maketitle
\begin{abstract} It was shown in 
[S.~Kaliman, M.~Zaidenberg,
\emph{Gromov ellipticity of cones over 
projective manifolds},
Math. Res. Lett. (to appear), arXiv:2303.02036 (2023)]
that the affine cones over flag manifolds and
rational smooth projective surfaces are elliptic 
in the sense of 
Gromov. 
The latter remains true after successive blowups 
of points on these varieties. 
In the present article we extend this to smooth projective 
spherical varieties (in particular, toric varieties) successively 
blown up along smooth subvarieties. 
The same holds, more generally, 
for uniformly rational  projective varieties,
in particular,
for projective varieties covered by affine spaces. It occurs also that 
stably uniformly rational complete varieties are elliptic. 
\end{abstract}
\section{Introduction}\label{sec:intro}
We work over algebraically closed field $\kk$ 
of characteristic zero. All the varieties 
in this paper are algebraic 
varieties defined over $\kk$;
$\PP^n$ and $\A^n$ stand for the projective 
resp. affine $n$-space 
over $\kk$;  $\GG_{\mathrm m}$ and $\GG_{\mathrm a}$ 
stand for the one-dimensional algebraic torus 
and the one-dimensional
unipotent algebraic group over $\kk$, respectively. 
All the notions, such as a neighborhood, a spray, etc. 
 are considered in the algebraic category unless otherwise noted. 
``Ellipticity'' below means ``Gromov's algebraic ellipticity''. 
\subsection{Gromov's ellipticity} 
The notion of Gromov ellipticity appeared first 
in analytic geometry 
where it serves in order to establish 
the Oka-Grauert Principle in the most general form, 
see \cite{Gro89} 
and \cite{For17}.  Besides, for an elliptic complex 
manifold $X$ the following approximation property 
holds: every holomorphic map $f\colon K \to X$ 
from a neighborhood of 
a compact convex set $K\subset\CC^n$ 
can be approximated by holomorphic maps 
$\CC^n\to X$, see \cite{Gro89}. 
A manifold $X$ with the latter property 
is called an \emph{Oka manifold}, see
the survey article \cite{For23}. 
If $X$ is algebraic and elliptic in the algebraic sense, then $f$ 
can be approximated by morphisms $\CC^n\to X$, 
see \cite[Corollary 6.5]{For23}. 

Gromov considered as well an analogous notion of ellipticity 
in the setup of algebraic varieties. It is known that 
an elliptic smooth
algebraic variety $X$ of dimension $n$
admits a surjective morphism from $\A^{n+1}$ 
which also is smooth  and surjective on
an open subset of $\A^{n+1}$, see \cite{Kus22a}. 
This implies that the endomorphism monoid
${\rm End}(X)$ is highly transitive on $X$, 
see \cite[Appendix~A]{KZ23b}.
Furthermore, for $\kk=\CC$
the fundamental group $\pi_1(X)$ is finite, 
see \cite{Kus22b}. 

Recall that a smooth algebraic variety $X$ is called 
\emph{elliptic} if it admits 
a dominating Gromov spray $(E,p,S)$ where
$p\colon E\to X$ is a vector bundle with zero section $Z$ 
and $s\colon E\to X$ is a morphism such that
$s|_Z=p|_Z$ and $s$ is dominating at any point $x\in X$, 
that is, the restriction $s|_{E_x}$ to the fiber $E_x=p^{-1}(x)$ 
is dominant at the origin $0_x\in E_x$. 
The image $O_x=s(E_x)\subset X$ 
is called the \emph{$s$-orbit of $x$}.

According to \cite[3.5.B]{Gro89} (see also 
\cite[Proposition 6.4.2]{For17}, \cite[Remark 3]{LT17} and 
\cite[Appendix B]{KZ23b})
if the ellipticity holds locally on an open covering of~$X$,
then it holds globally. Moreover, the ellipticity of $X$ holds 
if $X$ is subelliptic, that is, 
there is a dominating collection of sprays on $X$ 
instead of a single spray, see 
\cite[Definition 2.1]{For06}, 
\cite{For17} and \cite{KZ23a}. 
In other words, one can always replace 
a dominating collection of sprays on $X$ 
with a single  dominating spray. 
\subsection{Ellipticity of cones} 
Let $X\subset\PP^N$ be a smooth projective 
variety of dimension~$n$. 
The affine cone ${\rm cone}(X)$ in $\A^{N+1}$ 
blown up at the origin gives rise to
a line bundle $F=\cO_X(-1)$ on $X$ 
whose zero section $Z_F$ 
is the exceptional divisor of the blowup. 
The associated principal 
$\GG_{\mathrm m}$-fiber bundle $Y\to X$ 
with fiber $\A^1_*=\A^1\setminus\{0\}$
is isomorphic to $F\setminus Z_F$ and so,
to the punctured affine cone 
over $X$ that is, the affine cone with its vertex removed:
\[Y=F\setminus Z_F\simeq_X
 {\rm cone}(X)\setminus\{0\}.\]

Our aim is to establish the ellipticity of $Y$ 
provided $X$ is elliptic,  
under certain additional assumptions on $X$. 
In \cite{KZ23b} the second and the third authors 
suggested
criteria of ellipticity of $Y$ based on the 
so called \emph{curve-orbit property}
for some families of smooth rational curves 
and sprays on $X$. 
In particular,  it was shown in \cite{KZ23b} 
that the punctured affine cones over 
a flag variety $G/P$ blown up in several points 
and infinitesimally near points 
are elliptic and the same holds for any rational 
smooth projective surface, see \cite[Theorem 0.1]{KZ23b}. 
In the present note we develop 
further the technique of \cite{KZ23b}. This
allows us to establish  similar facts for uniformly rational varieties,
in particular, for varieties of 
class~$\cA_0$, therefore,
for smooth projective toric  and,
more generally,  spherical varieties.
Recall that a spherical variety is a normal $G$-variety 
which contains an open $B$-orbit, 
where $G$ is a reductive algebraic group and 
$B$ is a Borel subgroup in $G$. 
A flag variety $G/P$ and 
a normal toric variety are spherical varieties.  
For $G/P$ the latter follows 
from the Bruhat decomposition $G=BWB$, and
for a toric $T$-variety it suffices to choose $G=B=T$.
\subsection{Varieties of class $\cA_0$} 
One says that a variety $X$  
belongs to class $\cA_0$ 
if there is an open cover
$\{A_i\}$ on $X$
by affine cells $A_i\simeq\A^n$ where $n=\dim(X)$; {see 
\cite[Definition 2.3]{For06}.}\footnote{
A variety of class $\cA_0$ is also said to be 
\emph{$A$-covered}, see \cite[Definition 4]{APS14}. }
It is well known that 
a variety of class~$\cA_0$ is elliptic; 
see, e.g., \cite[Sec. 3.5]{Gro89}. 
The blowup of 
a variety of class $\cA_0$ in a point is again 
a variety of class~$\cA_0$, see \cite[3.5D]{Gro89}.  More generally, 
suppose $X$ is a variety of class $\cA_0$ and $Z\subset X$ 
is a closed subvariety
such that the pair $(A_i, Z\cap A_i)$ 
is isomorphic for any $i$  
to a pair $(\A^n, \A^k)$
with $n-k\ge 2$;  
in this case $Z$ is called a 
\emph{linear subvariety of $X$}.
The blowup of a linear subvariety $Z$ in $X$
results again in a variety of class $\cA_0$, 
see \cite[Section~4, Statement~9]{APS14}. 
\subsection{Uniformly rational varieties} 
This class of varieties strictly contains the class~$\cA_0$, 
see Examples \ref{exa:bb}.
\bdefi\label{uni.d1} An algebraic variety $X$ is called 
\emph{uniformly rational} \footnote{In other terms, regular, plain or
locally flattenable, 
see \cite[35.D]{Gro89}, \cite{BHSV08} and \cite{Pop20}, respectively.}
if for each $x\in X$ there is an open neighborhood $X_0$ of $x$
in $X$ isomorphic to an open subset of $\A^n$.
\edefi
In \cite[3.5.E$'''$]{Gro89}
Gromov asked whether a smooth 
complete rational variety is uniformly rational. 
It seems that this question is still open, see 
\cite[Question 1.1]{BB14} and \cite[p. 41]{CPPZ21}. 
On the other hand, not every complete uniformly rational 
variety belongs to class $\cA_0$.
For instance, none of the smooth rational cubic fourfolds in $\PP^5$ 
and none of the smooth 
threefold intersections of a pair of quadrics in $\PP^5$ contains 
a Zariski open set isomorphic to an affine space, see \cite{PS88}, 
and \cite{Pro94}. However, these varieties are uniformly rational, 
see \cite{BB14} and Examples \ref{exa:bb} below.

For the following property of uniformly rational varieties 
see \cite[Proposition~3.5E]{Gro89}, 
\cite[Theorem 4.4]{BHSV08} and \cite[Proposition 2.6]{BB14}.
\bthm\label{uni.t1} Let $X$ be a uniformly 
rational variety and $\tX\to X$ be
the blowing of $X$ up along a
smooth subvariety of codimension at least $2$. 
Then $\tX$ is uniformly rational.
\ethm
Notice that the total space of 
a locally trivial fiber bundle over a uniformly rational 
variety with a uniformly rational general fiber 
also is uniformly rational.
\subsection{Main results}
We prove the following theorem.
\begin{thm}[Theorem \ref{thm:main}] \label{mthm}
Let $X$ be a complete uniformly 
rational  variety of positive dimension. 
Then $X$ is elliptic. 
Let further $X$ be projective, $D$ be an ample  
divisor on $X$ and
$Y=F\setminus Z_{F}$ be the principal 
$\GG_{\mathrm m}$-fiber bundle 
associated with a line bundle~$F$ where
 either $F=\cO_X(-D)$ or $F=\cO_X(D)$.
Then $Y$ is elliptic. 
\end{thm}
\begin{rems} 1. 
Up to isomorphism over $X$ which
inverses the $\GG_{\mathrm m}$-action, the 
$\GG_{\mathrm m}$-variety
$Y=F\setminus Z_{F}$ stays the same under replacing 
$D$ by $-D$. 

2. For a trivial line bundle~$F$ on $X$
the variety $Y\simeq_X X\times (\A^1 \setminus\{0\})$ 
is not elliptic and 
$\pi_1(Y)$ is infinite for $\kk=\CC$.

3. If $\Pic(X)=\ZZ$ then any non-principal divisor $D$ 
on $X$ is either ample or anti-ample. 
 
4. As a simple example, consider $X=\PP^1$ and let
$D$ be a point of $\PP^1$. Then
$F=\cO_{\PP^1}(-1)$ is the tautological line bundle
on $\PP^1$
and $Y=\A^2\setminus\{0\}$, 
which is elliptic. 
For  $D=0$ we obtain 
$Y=\PP^1\times (\A^1 \setminus\{0\})$, 
which is not elliptic.
\end{rems}
From Theorems \ref{uni.t1} and \ref{mthm} we 
deduce the following fact.
\begin{cor}\label{cor:main}
The variety $X'$ resulting from a sequence of blowups 
of a complete uniformly 
rational  variety $X$ along
smooth subvarieties  is elliptic. 
\end{cor}
A closely related result in \cite[Corollary 2]{LT17}  says 
that the blowup $X'$ with a smooth center of codimension at least $2$ 
in a variety $X$ of class $\cA_0$ is subelliptic. 
Hence, $X'$ is elliptic by \cite[Theorem 0.1]{KZ23a}.
See also \cite{KKT18} for a similar result. 

Notice also that due to Theorem \ref{mthm} and 
to Gromov's theorem mentioned above, 
any uniformly 
rational compact complex 
algebraic variety $X$ is an Oka manifold.

Smooth complete 
 spherical varieties and smooth complete rational 
$T$-varieties of complexity one
belong to class $\cA_0$, see \cite{BLV86} and \cite{APS14}. 
Hence, these varieties are uniformly rational. So, 
we have the following corollary.
\begin{cor}[cf. Corollary \ref{cor:tor-sph}]
\label{cor:examples}
Let a smooth projective variety $X$ be  
either spherical or a rational $T$-variety of complexity one. 
Then the conclusions of Theorem~\ref{mthm} hold for $X$ 
successively blown up along smooth subvarieties. 
In particular, this holds if $X$ is a toric variety or a flag variety. 
\end{cor}
The next result concerns 
complete unirational varieties. 
\begin{thm}\label{cor:surj} 
Let $X$ be a complete unirational
variety of dimension $n$. 
Then there exists a smooth complete uniformly rational  
variety $\tX$ of dimension $n$
and surjective  morphisms $\A^{n+1}\to\tX\to X$. 
If $X$ is rational then the morphism $\tX\to X$ 
can be chosen to be birational. Furthermore, if the base
field $\kk$ is $\CC$, then there are surjective morphisms
$\A^{n}\to\tX\to X$. 
\end{thm}
\begin{proof} By Chow's Lemma 
there exists a projective 
variety $X'$ and a birational surjective morphism $X'\to X$, 
see \cite[Ch. II. Exercise 4.10]{Har04}. Clearly, 
replacing $X$ by $X'$ we may assume that $X$ is projective. 

Choose a generically finite dominant rational map 
$h\colon\PP^n\dasharrow X$, which is birational if $X$ is rational.
By Hironaka's theorem on elimination 
of indeterminacy
there exists a commutative diagram 
\usetikzlibrary{matrix,arrows,decorations.pathmorphing}
     \begin{center}
        \begin{tikzpicture}[scale=2]
        
        \node at (1,1){$\tX$};
        \node at (0,0){$\PP^n$};
        \node at (2,0){$X$};
        \node at (1,0.2){$h$};
        \draw[->][] (0.8,0.8)--node[above=1pt]{$f$} (0.2,0.2); 
         \draw[->][] (1.1,0.8)--node[above=1pt]{$g$} (1.8,0.2); 
        \draw[->][thick,dashed] (0.2,0)--(1.8,0);
       
        \end{tikzpicture}
\end{center}
where $f$ is a composition of blowups with smooth 
irreducible centers and $g$ is a generically finite morphism, 
which is birational if $h$ is,
see \cite{Hir64} and
\cite[Corollary 3.18 and Theorem 3.21]{Kol07}. 
By
Theorem \ref{uni.t1}, $\tX$ is uniformly rational,
hence elliptic, 
see Theorem \ref{mthm}. This allows to apply
a theorem of Kusakabe \cite{Kus22a}
which says that there is a surjective morphism $\A^{n+1}\to \tX$. 
Moreover, if $\kk=\CC$, then there is a surjective 
morphism $\A^{n}\to\tX$ by a result of Forstneri\v{c}, see 
\cite[Theorem 1.6]{For17a}. 
\end{proof}
From Theorem \ref{cor:surj} we deduce the following 
characterization of unirationality:
\begin{cor}
A complete variety 
$X$ over $\kk$ of dimension $n$ is unirational 
if and only if $X$ admits a surjective morphism from 
$\A^{n+1}$ (resp., from $\A^n$ if the base field is $\CC$).
\end{cor}
\begin{rems} 
1. 
The assumptions of completeness in Theorems \ref{mthm} and
\ref{cor:surj} are important, as the following simple example shows:  
take
the complement $T$ of the coordinate cross $xy=0$ in $\CC^2$. 
See, however, \cite{Arz23} and \cite{Bar23} for 
certain classes of affine and quasiaffine varieties, respectively,
that are images of affine spaces. 

2. 
Recall Gromov's question (see \cite[3.5$''$]{Gro89}): is every
unirational smooth complete variety elliptic?

3.
By the Lefschetz Principle, the theorem of
Forstneri\v{c} cited above holds over any algebraically closed 
field of characteristic zero which has
infinite transcendence degree 
over $\QQ$, see \cite{Ekl73}. 

4. 
The version of Kusakabe's theorem used 
in the above proof
says that a smooth complete 
elliptic variety $\tX$ admits 
a morphism $\A^{n+1}\to \tX$. Moreover, 
this morphism could be chosen so that its restriction 
to an open subset of $\A^{n+1}$ is smooth and surjective. 
For the reader's convenience we sketch a short 
argument close to the original one in \cite{Kus22a}.
Notice that the original Kusakabe's theorem works 
for a not necessarily complete smooth elliptic variety. 

Fix a dominating spray 
$(E,p,s)$ on $\tX$ of rank 
$r\ge n$. For any point $x\in \tX$ 
the restriction $s|_{E_x}\colon E_x\to \tX$
has surjective differential at the origin $0_x\in E_x$. 
Hence, one can choose a 
vector subspace $F_x\subset E_x$ of dimension $n$ 
such that 
$s|_{F_x}\colon F_x\to \tX$ is \'etale at $0_x\in F_x$. 
It follows that $F_x$ contains an open neighborhood 
$V_x$ of $0_x\in F_x$ 
such that the differential of $s|_{V_x}$  
has rank $n$ at any point $v\in V_x$
and $s(V_x)$ contains a Zariski open dense 
neighborhood $U_x$ 
of $x$ in $\tX$.  
Choosing a finite open covering $\{U_{x_i}\}$ 
of $\tX$, $i=1,\ldots,k$ one has $s(V)=\tX$ 
where $V=\cup_{i} V_{x_i}$. 

Since $\tX$ is unirational, 
hence rationally connected, 
there exists a rational curve $C$ in $\tX$ 
which passes through 
$x_1,\ldots,x_k$, see \cite[Ch. IV, Theorem 3.9]{Kol96}. 
The normalization morphism 
$\eta\colon \PP^1\to C$ induces 
a vector bundle $\eta^*E\to\PP^1$. 
Choose an affine chart 
$A\simeq\A^1$ in $\PP^1$ which contains 
$\eta^{-1}(x_i)$ for $i=1,\ldots,k$. The restriction of 
$\eta^*E|_A$ is trivial: $\eta^*E|_A\cong_A 
\A^1\times\A^r=\A^{r+1}$. 
Let $\tilde V$ resp. $\tilde F$ be the preimage of $V$ resp. 
of $F=\cup_{i} F_{x_i}$ in $\eta^*E|_A$. Identifying 
$\eta^*E|_A$ with the trivial vector bundle 
$\A^1\times\A^r\to\A^1$ one can find 
an automorphism $\phi$ of the latter
identical on $\A^1$
which 
sends every $\A^n$-component of $\tilde F$
to the fixed  $\A^n$-subspace of $\A^r$. 
Thus, $\phi$ sends $\tilde F$ into 
$\A^1\times\A^n=\A^{n+1}$. 
Since $s(V)=\tX$, letting $\tilde s=s\circ\eta_*\circ\phi^{-1}$
one has $\tilde s(\tilde V)=\tX$, 
where $\eta_*\colon \eta^*E|_A\to E|_A$ 
is the natural surjective morphism induced by $\eta$. 
Moreover, there exists an open neighborhood $\Omega$ of 
$\tilde V$ in $\A^{n+1}$ 
such that $\tilde s(\Omega)=\tX$ and 
$\tilde s|_\Omega\colon\Omega\to \tX$ is smooth.
\qed
\end{rems}
\subsection{Generalized affine cones}
By Theorem \ref{mthm} the punctured affine cones 
over uniformly rational  projective varieties
are elliptic. Let us mention further examples. 

Recall that an effective $\GG_{\mathrm m}$-action 
$\lambda\colon \GG_{\mathrm m}\times\bar Y\to\bar Y$ 
on a normal affine variety 
$\bar Y$ is called \emph{good}  
if there exists 
a point $y_0\in \bar Y$ 
which belongs to the closure of any \mbox{$\lambda$-orbit}.
The structure algebra $A=\cO_{\bar Y}(\bar Y)$ 
of such a $\GG_{\mathrm m}$-variety $\bar Y$
is positively graded:
$A=\bigoplus_{k\ge 0} A_k$  
where $A_0=\kk$ and $A_k$ for $k>0$ consists
of $\lambda$-homogeneous elements of weight $k$. 
Let $Y=\bar Y\setminus\{y_0\}$ and 
$X={\rm Proj}(A)=Y/\lambda$. 
Then $X$ is a normal projective variety, 
see \cite[Proposition 3.3]{Dem88}. 
According to \cite[Theorem 3.5]{Dem88}, 
see also \cite[Theorem 3.3.4]{Dol07},
there exists
an ample $\QQ$-Cartier divisor 
$D=\sum_i p_i/q_i D_i$ on $X$, where the $D_i$ 
are prime divisors and the integers $p_i,q_i$ are coprime,
such that 
\[A_k=H^0(X,\cO_X(\lfloor kD\rfloor))\quad\text{for 
every}\quad k\ge 0.\] 
Furthermore, the $\GG_{\mathrm m}$-action 
$\lambda$ on $Y$ is free if and only if $D$ 
is a Cartier divisor that is, $q_i=1$ $\forall i$, 
see \cite[Corollaire 2.8.1]{Dem88}. 

Conversely, given a smooth projective variety $X$ 
and an ample Cartier divisor
$D$ on $X$ 
one can consider the generalized affine cone
\[\bar Y=\Spec\left(\bigoplus_{n=0}^\infty 
H^0\left(X,\cO_X(nD)\right)\right).\]
This is a normal affine variety equipped with 
a good $\GG_{\mathrm m}$-action,
see \cite[Sec.~3]{Dem88} or \cite[Proposition 3.3.5]{Dol07}. 
Letting 
$Y=\bar Y\setminus\{y_0\}$ 
where $y_0\in \bar Y$ is
the unique \mbox{$\GG_{\mathrm m}$-fixed} point, one gets 
a morphism $\pi\colon Y\to X=Y/\GG_{\mathrm m}$. 
Every fiber of $\pi$ is reduced, irreducible 
and isomorphic to $\A^1_*$, see 
\cite[Proposition 2.8]{Dem88} and 
\cite[Proposition~3.4.5]{Dol07}. 
Furthermore, the $\GG_{\mathrm m}$-action on $Y$ is free 
and $\pi$ is locally trivial, 
see \cite[the proof of Proposition 2.8]{Dem88}.

Consider also the line bundle 
$F=\cO_X(-D)\to X$ equipped with 
the associated $\GG_{\mathrm m}$-action. 
We have a birational morphism 
$F\to \bar Y$ contracting the zero section 
$Z_F$ to a normal point $y_0\in \bar Y$, 
cf. \cite[3.4]{Dem88}. 
It restricts to an equivariant isomorphism
of smooth quasiaffine varieties
\[F\setminus Z_{F}\simeq \bar Y\setminus\{y_0\}\] 
equipped with free $\GG_{\mathrm m}$-actions, 
see \cite[Corollaire 2.9]{Dem88}, \cite[Sec. 3.4, p. 49]{Dol07}
and \cite[Sec. 1.15]{KPZ13}; cf. \cite[p. 183]{Pin77}.

Notice that while $\bar Y$ above is normal,  
for $X\subset \PP^N$ 
the affine cone ${\rm cone}(X)$ is normal if and only if 
the embedding $X\hookrightarrow \PP^N$
is projectively normal, that is,  for every $d\ge 1$
the linear system cut out on $X$ 
by the hypersurfaces of degree $d$ 
is complete,
see \cite[Chap. II, Example 7.8.4]{Har04}. 
Thus, if $D$ is a hyperplane sections of 
$X\subset \PP^N$ 
then $\bar Y$ as above is the normalization 
of the affine cone
${\rm cone}(X)$. In particular, the punctured cone
${\rm cone}(X)\setminus\{0\}$ is 
$\GG_{\mathrm m}$-equivariantly isomorphic to 
$\bar Y\setminus\{y_0\}$.
\section{The curve-orbit property}
\label{ss:alter} 
We need the following more general analog 
of the curve-orbit property $(*)$
for $\GG_{\mathrm a}$-sprays defined in 
\cite[Definition 2.7]{KZ23b}. 
\begin{defi}\label{def:like} 
Given a smooth variety 
$B$ of dimension $n-1$ 
consider the $\GG_{\mathrm a}$-action on the 
cylinder $V=B\times\A^1$ by 
shifts on the second factor:
\[s_V\colon \GG_{\mathrm a}\times V\to V,\quad
 (t, (b,v))\mapsto (b,v+t)\]
along with the associated \emph{$\GG_{\mathrm a}$-spray} 
$(E_V,p_V,s_V)$ on $V$ where
$E_V=V\times\A^1$ and $p_V\colon E_V\to V$ 
is the first projection. 

Let $X$  be a smooth variety of dimension $n$. 
Assume that $X$ admits a birational morphism
$\psi\colon V\to X$ biregular 
on an open dense subset $V_0\subset V$ with image 
$X_0\subset X$. Consider the spray $(E_0,p_0,s_0)$ on $X_0$ 
with values in $X$
conjugate to $(E_V,p_V,s_V)|_{V_0}$ via~$\psi$. That is, 
$E_0=X_0\times\A^1$, $p_0\colon E_0\to X_0$ 
is the first projection
and 
\[s_0\colon E_0\to X,\quad (x,t)\mapsto 
\psi(s_V(t,\psi^{-1}(x))).\]
Extend $(E_0,p_0,s_0)$ to a rank 1 spray $(E,p,s)$ on $X$;
 the latter spray exists due to Gromov's Extension Lemma, 
 see \cite[3.5B]{Gro89},  \cite[Propositions 6.4.1-6.4.2]{For17} 
 or \cite[Proposition 8.1]{KZ23b}. 
 We call $(E,p,s)$ a \emph{$\GG_{\mathrm a}$-like spray on $X$}.
This spray is associated with the birational 
$\GG_{\mathrm a}$-action on $X$
 conjugate via $\psi$ to the standard $\GG_{\mathrm a}$-action 
 on the cylinder~$V$, see \cite[Chap. 1]{Dem70}. 
\end{defi}
\begin{rem}\label{rem:diff} 
The closure $C_x=\overline{O_x}$ in $X$ 
of the $s_0$-orbit $O_x$ 
of a point $x\in X_0$ is a
rational curve. 
By construction, the intersection $C_x\cap X_0$ 
is smooth and $(E,p,s)$ restricts to 
a dominating spray on $O_x\cap X_0$, 
cf.\ Lemma \ref{lem:Gm-like} below. Moreover, 
the morphism $s\colon E_x\simeq\A^1\to C_x$ 
admits a lift to  the normalization $\PP^1$ of $C_x$, 
and the latter morphism $\A^1\to\PP^1$
is an embedding. 

However, the curve $C_x$ 
can have singularities off $X_0$. 
Thus, the setup of Definition \ref{def:like} 
does not guarantee  that $X$ verifies on $X_0$ 
either the curve-orbit property $(*)$
of \cite[Definition 2.7]{KZ23b}, or the enhanced 
curve-orbit property $(**)$ of \cite[Proposition 3.1]{KZ23b}.
Indeed, the latter properties postulate the smoothness 
of $C_x$ for $x\in X_0$, which 
occurs to be a rather restrictive condition 
for our purposes. 
The question arises whether any smooth 
complete variety of class $\cA_0$, 
or even every complete uniformly rational variety,  
verifies the  curve-orbit property $(*)$ of \cite{KZ23b} 
with smooth  rational curves; cf. 
Corollary \ref{cor:2-orbits}  and
Lemma \ref{lem:data} below.
Notice that this property holds for smooth
complete  rational surfaces and for flag varieties 
$G/P$, see \cite{KZ23b}.
\end{rem}
In order to use a criterion of ellipticity 
of cones over projective varieties from 
\cite[Corollary 2.9]{KZ23b} 
we introduce the following objects. 
\begin{defi}\label{def:like-2} 
Let $X$ and $B$ be as in Definition \ref{def:like}.
Consider a $\PP^1$-cylinder $W=B\times\PP^1$ with base $B$.
Assume that $X$ admits a birational morphism
$\phi\colon W\to X$ biregular 
on an open dense subset $W_0\subset W$ with image 
$X_0\subset X$. Given a point $u\in \PP^1$ consider the 
cylinder $V_u:=B\times 
(\PP^1\setminus\{u\})\simeq B\times\A^1$,
the birational morphism 
$\psi_u=\phi|_{V_u}\colon V_u\to X$
and the open dense subsets 
$V_u\cap W_0\subset V_u$ and 
$X_u=\psi_u(V_u\cap W_0)\subset X_0$.
Thus, the data $(W,\phi,X_0)$ yields 
a one-parameter family of 
$\GG_{\mathrm a}$-like sprays $(E_u,p_u,s_u)$ on $X$ 
where $u\in\PP^1$, see Definition \ref{def:like}. 
\end{defi}
\begin{lem}\label{lem:Gm-like} 
Under the setup of Definition \ref{def:like-2} 
 let for  $x\in X_0$,
\begin{equation}\label{eq:Cx} 
w=\phi^{-1}(x)=(b, u_x)\in W_0\quad\text{and}\quad 
C_x=\phi(\{b\}\times\PP^1)\subset X.
\end{equation} 
Then $C_x$ is a complete rational 
curve in $X$ through $x$ 
such that $C_x\cap X_0$ 
is smooth. 
If $x\in X_u$, that is $\phi(b,u)\neq x$, then the 
$s_u$-orbit $O_{u,x}$ of $x$ is one-dimensional 
and $(E_u,p_u,s_u)$ restricts to 
a spray on $C_x$ 
dominating at $x$ and such that 
$s_u|_{E_{u,x}}\colon E_{u,x}\to O_{u,x}$ 
is a birational morphism \'etale over $x$.
\end{lem}
\begin{proof} 
The $\GG_{\mathrm a}$-like spray $(E_u,p_u,s_u)$
inherits a kind of the composition property 
of a $\GG_{\mathrm a}$-action. Namely, 
for any $x'\in O_{u,x}\cap X_0$ the $s_u$-orbits 
$O_{u,x'}$ and $O_{u,x}$
coincide. This implies that $(E_u,p_u,s_u)$ restricts to 
a spray on $C_x$.
The rest of the proof is easy and is left to the reader.
\end{proof}
\begin{defi}\label{def:co-property}
Modifying \cite[Definition 2.7]{KZ23b} we say that a 
complete rational curve $C$ on a smooth variety $X$
verifies the  \emph{strengthened two-orbit property} 
at a smooth point $x\in C$  
if 
\begin{itemize}
\item[$(*)$] there exists a pair  of rank 1 
sprays $(E_i,p_i,s_i)$ $(i=1,2)$ 
on $X$ such that $C$ is covered by the 
one-dimensional $s_i$-orbits 
$O_{i,x}$, $s_i\colon E_{i,x}\to  O_{i,x}$ 
is a birational morphism \'etale over $x$
 and $(E_i,p_i,s_i)$ restricts to a spray 
on $O_{i,x}$ dominating at~$x$.
\end{itemize}
If for any $x\in X$ there exists a curve 
$C=C_x$ as above, 
then we say that $X$ verifies the 
\emph{strengthened curve-orbit property}.
\end{defi}
\begin{rem}\label{rem:without-restriction}
Following the lines of the proof of Proposition 
\ref{prop:ell-cone} one can establish 
the ellipticity of the punctured affine cone over 
an elliptic smooth projective variety with 
an ample polarization under the following
weaker assumption:
\begin{enumerate}
\item[$(*')$]
for each $x\in X$ 
there exists a rational curve $C_x$ in $X$ 
and a pair of sprays 
$\{(E_i,p_i,s_i)\}_{i=1,2}$ on $X$ such that 
\begin{itemize} 
\item $x$ is a smooth point of $C_x$;
\item $s_i|_{E_{i,x}}\colon E_{i,x}\to  O_{i,x}$ 
is a birational morphism \'etale over $x$;
\item $C_x=O_{1,x}\cup O_{2,x}$. 
\end{itemize}
\end{enumerate}
However, in the concrete setup of the present paper 
the strengthened curve-orbit property $(*)$ holds as well.
\end{rem}
We have the following corollary.
\begin{cor}\label{cor:2-orbits} 
Let $X$ be a smooth projective variety, and let 
$(W,\phi, X_0)$ be a data 
as in Definition \ref{def:like-2}. For
$x\in X_0$  let $C_x$
be a curve as in \eqref{eq:Cx}. Then
$C_x$ verifies the strengthened two-orbit property $(*)$ at $x$
with a pair of $\GG_{\mathrm a}$-like sprays. 
If for each $x\in X$ there exists a data  $(W,\phi, X_0)$
such that $x\in X_0$ then the strengthened curve-orbit property 
holds on~$X$ with pairs of $\GG_{\mathrm a}$-like sprays. 
\end{cor}
\begin{proof} 
Let $\phi^{-1}(x)=(b,u_x)\in W_0$. 
Pick two distinct points $u_1,u_2\in\PP^1$ 
different from 
$u_x$ and consider the corresponding 
$\GG_{\mathrm a}$-like sprays 
$(E_i,p_i,s_i)=(E_{u_i},p_{u_i},s_{u_i})$, $i=1,2$, 
see Definition \ref{def:like-2}. 
Due to Lemma \ref{lem:Gm-like} 
these sprays fit in  
Definition \ref{def:co-property}
of the strengthened two-orbit property.
This yields the first assertion. 
Now the second is immediate. 
\end{proof}
Using Definitions \ref{def:like} and \ref{def:like-2} 
we can generalize the ellipticity criterion 
for cones in
\cite[Corollary 2.9]{KZ23b} as follows. 
The proof repeats verbatim the proof 
of Corollary~2.9 in [KZ23b]
with minor changes.
\begin{prop}\label{prop:ell-cone}  
Let $X$ be a smooth projective variety and 
$\rho\colon F\to X$ 
be an ample line bundle. 
Suppose that $X$ is elliptic 
and 
for any point $x\in X$ there exists a data 
$(W,\phi, X_0)$ as in Definition \ref{def:like-2} 
such that $x\in X_0$.
Then $Y=F\setminus Z_F$ is elliptic. 
\end{prop}
\begin{proof} 
A dominating spray $(E,p,s)$ on $X$ 
lifts to a spray $(\hE,\hp,\hs)$ of rank $n=\dim(X)$ on $Y$
where $\hE$ fits in the commutative diagrams
\begin{equation}\label{eq:1}
\begin{array}{ccc} \hE &  \stackrel{{\hp}}{\longrightarrow} & Y\\
\, \, \, \, \downarrow^{\hat \rho}   & 
& \, \, \, \,\,\, \downarrow^{\rho|_Y}\\
E &  \stackrel{{p}}{\longrightarrow} & X 
\end{array} 
\quad\text{and}\quad 
\begin{array}{ccc} \hE &  \stackrel{{\hs}}{\longrightarrow} & Y\\
\, \, \, \, \downarrow^{\hat \rho}   & 
& \, \, \, \,\,\, \downarrow^{\rho|_Y}\\
E &  \stackrel{{s}}{\longrightarrow} & X 
\end{array} 
\end{equation}
see \cite[Lemma 2.3]{KZ23b}.

Given $y\in Y$ we let $x=\rho(y)\in X$ and let 
$C_x$ be a rational curve 
on $X$ passing through $x$ which is smooth at $x$ 
and such that $(C_x,x)$ satisfies the strengthened two-orbit property 
with a pair of $\GG_{\mathrm a}$-like sprays $(E_i,p_i,s_i)$ 
on $X$, $i=1,2$, see Definition \ref{def:co-property}
and Corollary \ref{cor:2-orbits}.  
Due to \cite[Lemma 2.3]{KZ23b} these sprays
admit a lift to a pair of rank~$1$ sprays 
$(\hE_i,\hp_i,\hs_i)$ on $Y$.

We will show, following the lines 
of the proof of Proposition 2.6  in
\cite{KZ23b}, that the triplet of sprays 
$(\hE,\hp,\hs)$ and $(\hE_i,\hp_i,\hs_i)$, 
$i=1,2$ is dominating at $y$. 
This implies the assertion. 

Notice that the tangent space at $y$ 
to the $\hs$-orbit is a hyperplane $H\subset T_yY$ such that
$d\rho(H)=T_xX$. 
We claim that the pair of tangent vectors at $y$ 
to the $\hs_i$-orbits $\hO_{i,y}$ on $Y$ span a plane 
$P$ in $T_yY$
such that $d\rho(P)=T_xC_x$.
Accepting this claim, 
there exists a nonzero vector
$v\in P$ such that $d\rho(v)=0$. Hence $v\notin H$ and so,
${\rm span}(H,P)= T_yY$, which gives the desired domination. 

To show the claim, consider a morphism of normalization
$\phi_x\colon \PP^1\to C_x$ and 
the pullback line bundle 
$\tF=\tF(x)=\phi_x^*(F|_{C_x})\to\PP^1$. Since $F\to X$ is ample  
one has $[F]\cdot [C_x]\neq 0$. 
Hence, $\tF\to\PP^1$
is nontrivial and so is the associated principal 
$\GG_{\mathrm m}$-fiber bundle 
$\tY=\tY(x)=\phi_x^*(Y|_{C_x})\to\PP^1$, 
see, e.g., \cite[Proposition 4.1]{Mit01}. For $i=1,2$ 
consider 
the pullback $\tp_i\colon \tE_{i}\to \tY$ 
of the line bundle $\hp_i\colon\hE_i\to Y$ 
via the induced morphism $\tY\to Y|_{C_x}$. 

According to the second diagram in \eqref{eq:1} we have 
\[\rho(\hO_{i,y})=\rho\circ\hs_i(\hE_{i,y})=
s_i(E_{i,x})=O_{i,x}\subset C_x.\]
 The birational morphism 
$s_{i}|_{E_{i,x}}\colon E_{i,x}=\A^1\to  C_x$ admits a lift 
 to normalization $\ts_{i,x}\colon \A^1\to\PP^1$. 
In fact, $\ts_{i,x}$ is a birational morphism smooth at $0$.
Hence, $\ts_{i,x}$ sends $\A^1$
isomorphically onto $U_i:=\PP^1\setminus\{u_i\}$,
the notation being as in the proof of Corollary \ref{cor:2-orbits}.

By Lemma \ref{lem:Gm-like}  $(E_i,p_i,s_i)$ restricts to a 
$\GG_{\mathrm a}$-like spray on $C_x$. It is easily seen that 
there is a pullback $\phi^*_x((E_i,p_i,s_i)|_{C_x})$ 
to a $\GG_{\mathrm a}$-like spray $(E'_i,p'_i,s'_i)$ on $\PP^1$
dominating at~$x$ and such that the $s'_i$-orbit of $x$
coincides with  $U_i$. Now, the pullback of $(E'_i,p'_i,s'_i)$ to
$\tY$ via $\tY\to\PP^1$ gives a  
$\GG_{\mathrm a}$-like spray $(\tE_i,\tp_i,\ts_i)$ on $\tY$ 
with $\tp_i\colon\tE_i\to\tY$ as above.

There are trivializations
$\tY|_{U_i}\cong_{U_i} U_i\times\A^1_* $, $ i=1,2$.
 The standard $\GG_{\mathrm a}$-action on 
 $U_i\simeq\A^1$ lifts to 
 a $\GG_{\mathrm a}$-action on 
 $\tY|_{U_i}$ whose orbits are the constant sections. 
 Since any morphism $\A^1\to\A^1_*$ is constant, 
the one-dimensional $\ts_i$-orbits in $\tY$ also 
are constant sections. 
Since the $\GG_{\mathrm m}$-fiber bundle 
$\tY\to\PP^1$ is nontrivial, it admits no global section.
Hence, in an appropriate affine coordinate $z$  on $U_1\cap U_2$ 
the transition function equals  
$z^k$ with~$k\neq 0$. It follows that 
the constant sections over $U_1$ 
meet transversally the ones over~$U_2$.  

The normalization morphism $\phi\colon \PP^1\to C_x$ 
is \'etale over $x$.
Hence, also the morphism $\tY\to Y|_{C_x}$ is \'etale over $y$.
Finally, the $\hs_i$-orbits $\hO_{1,y}$ and $\hO_{2,y}$
meet transversally at $y$. This proves our claim. 
\end{proof}
\section{Uniformly rational varieties}
\subsection{The ellipticity and the curve-orbit property 
of complete uniformly rational varieties}
In order to apply Proposition \ref{prop:ell-cone}
to complete (e.g.., projective) varieties
we need the following Lemmas \ref{lem:data}--\ref{lem:ell}.
\begin{lem}\label{lem:data} 
Let $X$ be a smooth complete variety of dimension $n\ge 2$. 
Then $X$ is uniformly rational if and only if for any point $x\in X$ 
there exists a data 
$(W,\phi, X_0)$ where $W=B\times\PP^1$ 
is a cylinder over an
open set $B\subset\A^{n-1}$ and $\phi\colon W\dasharrow X$ 
is a birational map which sends biregularly an open subset 
$W_0\subset B\times\A^1\subset W$ 
onto a neighborhood $X_0$ of $x$ in $X$; 
cf. Definitions \ref{def:like} and \ref{def:like-2}.
Furthermore, if $X$ 
is uniformly rational, then one can choose
a $\GG_{\mathrm a}$-like spray $(E_u,p_u,s_u)$ 
as in Definition \ref{def:like-2} so that
the differential ${\rm d} s_u$ sends
$T_{0_x}E_{u,x}$ to a general line in $T_xX$.
\end{lem}
\begin{proof} The ``if'' part is immediate.
 To show the ``only if'' part, suppose $X$ 
 is uniformly rational. 
 Then for any $x\in X$ 
there is an open subset
$V_0\subset\A^n$
and an isomorphism $h_0$ of $V_0$ onto a neighborhood of $x$ in $X$. 

Embed $\A^n\hookrightarrow \PP^n$ 
and let $h \colon \PP^n \dashrightarrow X$ 
be the birational extension of $h_0$.
Let $v=h_0^{-1}(x)\in V_0$. 
Fix a general point $P\in H=\PP^n\setminus \A^n$ and let $\cL$ 
be the family of projective lines in $\PP^n$ which pass
through $P$.
The projection $\pi_P\colon\PP^n\dasharrow \PP^{n-1}$ 
with center $P$ restricts to
a linear projection 
$\tau \colon \A^n\to \A^{n-1}$. The latter defines a decomposition 
$\A^n\cong\A^{n-1}\times\A^1$ 
and a family of parallel affine lines $l_y=\{y\}\times\A^1$
in $\A^n$ where $y\in\A^{n-1}$ and $\bar l_y\in\cL$. 
We may assume that $\bar{l}_0\in\cL$ passes through $v$. 

The embedding
$\A^n=\A^{n-1}\times\A^1\hookrightarrow \PP^n$ extends to a birational morphism 
$\psi \colon \A^{n-1}\times\PP^1\to \PP^n$ which contracts 
$\A^{n-1}\times \{\infty\}$ to $P$. The birational map 
$\tilde\varphi:=h \circ \psi \colon\A^{n-1}\times\PP^1\dasharrow X$ 
fits 
in the lower triangle of the commutative diagram

\usetikzlibrary{matrix,arrows,decorations.pathmorphing}
     \begin{center}
        \begin{tikzpicture}[scale=2]
        
        \node at (1,1){$\hX$};
        \node at (0,0){$\PP^n$};
        \node at (2,0){$X$};
        \node at (1,0.2){$h$};
        \draw[->][] (0.8,0.8)--node[above=1pt]{$f$} (0.2,0.2); 
         \draw[->][] (1.1,0.8)--node[above=1pt]{$g$} (1.8,0.2); 
        \draw[->][thick,dashed] (0.2,0)--(1.8,0);
        \node at (1,-1){$\A^{n-1}\times\PP^1$};
        \draw[->][] (0.8,-0.8)--node[above=1pt]{$\psi$} (0.2,-0.2); 
         \draw[->][thick,dashed] (1.1,-0.8)--node[above=1pt]{$\tilde\phi$} (1.8,-0.2); 
       
 \end{tikzpicture}
\end{center}
 In the upper triangle, $\hat X$ stands 
for the resolution of the closure $\Gamma_h$
of the graph of $h$ in $\PP^n\times X$, while $f$ and $g$ 
stand for the standard projections 
of $\Gamma_h$ to the factors composed 
with the resolution morphism $\hX\to\Gamma_h$. 
Thus, $\hX$ is a smooth projective variety 
and $f$ and $g$ are birational morphisms. 

We claim that one can choose 
an open neighborhood $B$ 
of $\tau(v)=0 \in \A^{n-1}$ and an open subset 
$W_0\subset W:=B\times\PP^1$ which contains $v$ 
so that the restriction 
$\varphi=\tilde\varphi|_{W_0}\colon W_0\to X$
 is biregular onto its image 
$X_0=\varphi(W_0)\ni x$, as desired.

To show the claim notice that the inverse 
$f^{-1}$ of the birational morphism 
$f\colon \hX\to\PP^n$ between 
smooth projective varieties is a blowup of $\PP^n$ 
whose center is an ideal sheaf supported 
on a closed subvariety $Z\subset \PP^n$ 
of codimension at least 2, 
see \cite[Ch. II, Theorem 7.17]{Har04}.  
On the other hand, $Z$ is the indeterminacy locus of 
$h=g\circ f^{-1}$ (\cite[Ch. II, Proposition 7.13(b)]{Har04}). 
Thus, $h$ is regular on 
$\PP^n\setminus Z$ and so, $\tilde\phi=h\circ \psi$
is regular on $(\A^{n-1}\times\PP^1)\setminus \psi^{-1}(Z)$.

Since $P\in H$ is a general point, 
we may assume that $P\notin Z$. 
Thus, $h$ is regular in $P$ and 
$\tilde\phi=h\circ\psi$ is regular on 
$\A^{n-1}\times\{\infty\}$ (recall that 
$\psi(\A^{n-1}\times\{\infty\})=P$). Since $P\notin Z$ 
the image $\pi_P(Z)$ 
is a closed subvariety of $\PP^{n-1}$. 
Since $\codim_{\PP^n} Z\ge 2$ 
the image
$\Delta:=\tau(Z\setminus H)=\pi_P(Z)\setminus\pi_P(H)$
is a proper closed subvariety of  
$\A^{n-1}=\PP^{n-1}\setminus\pi_P(H)$. 
Since $P$ is a general point of $H$ and 
$\codim_{\PP^n} Z\ge 2$ we may 
suppose that $\bar l_0\cap Z=\emptyset$ and so, 
$\tau(v)\notin\Delta$. Since by our construction $h$ 
is regular on $V_0$ we have
$Z\cap V_0=\emptyset$.

Let $B=\A^{n-1}\setminus\Delta$.  
It follows from the preceding that
$\tilde\phi$ is regular on 
$B\times\A^1$. Furthermore, 
$\tilde\phi$ is regular on
$W=B\times\PP^1$ and
sends biregularly the neighborhood
$W_0:=\psi^{-1}(V_0)\subset W$ of $\psi^{-1}(v)$
onto the neighborhood $X_0:=\tilde\phi(W_0)$ of $x$ in $X$.
Letting $\phi=\tilde\phi|_{W}$ the  claim follows. 
This proves the first assertion of the lemma.

To show the second assertion we let
$u=\infty\in\PP^1$ so that
$v\in V_u=B\times\A^1$, 
see Definition \ref{def:like-2}. 
Via the conjugation with~$\phi$, the 
$\GG_{\mathrm a}$-action on $V_u$ 
by shifts on the factor $\A^1$ yields a
$\GG_{\mathrm a}$-like  spray $(E_u,p_u,s_u)$ on $X$ 
as in Definition \ref{def:like}. 
The restriction of $s_u$ to 
$E_{u,x}$ is an immersion 
at the origin $0_x\in E_{u,x}$, see Lemma \ref{lem:Gm-like}.

Since $\tau \colon \A^n \to \A^{n-1}$ 
is a general linear projection, the tangent vector at $v$ 
to the orbit of $v$ under 
the $\G_a$-action on $V_u$ 
is a general vector in $T_{v}V_0$.
It follows that the differential  ${\rm d} s_u$ sends
$T_{0_x} E_{u,x}$
to a general line in $T_{x}X$. 
\end{proof}
\begin{lem}\label{lem:ell}
A complete uniformly rational 
variety $X$ is elliptic. 
\end{lem}
\begin{proof}
By Lemma \ref{lem:data}
one can find $n$ different rank 1 sprays 
$(E_i,p_i,s_i)$ on $X$, $i=1,\ldots,n$ where $n=\dim(X)$  such that 
the lines ${\rm d} s_i(T_{0_{i,x}} E_{i,x})$ 
span the tangent space $T_{x}X$. 
The composition of these sprays gives 
a rank $n$ spray on $X$ dominating at $x$, 
see \cite[Corollary 2.2]{KZ23a}.
This implies that $X$ is locally elliptic, hence elliptic, 
see \cite[Theorem 1.1]{KZ23a}. 
\end{proof}
\subsection{The main results} 
We can now deduce our main result.
\begin{thm}\label{thm:main} For a complete uniformly 
rational variety $X$ the following holds.
\begin{enumerate}
\item[{\rm (i)}] $X$ is elliptic.
\item[{\rm (ii)}] Assume that $X$ is projective 
and let $F\to X$ be an ample or 
anti-ample line bundle on $X$
with zero section $Z_F$.
Then $Y=F\setminus Z_F$ is elliptic. 
\end{enumerate}
\end{thm}
\begin{proof}
Statement (i) follows from Lemma \ref{lem:ell} and (ii) 
follows from Proposition \ref{prop:ell-cone} 
due to Lemma \ref{lem:data}.
\end{proof}
In Corollary \ref{cor:lar} below we slightly 
generalize statement (i). 
Let us introduce the following notion. 
\begin{defi}
We say that a variety $X$ is \emph{stably 
uniformly rational} 
if for some $k\ge 0$ the variety $X\times \A^k$ 
is uniformly rational. 
\end{defi}
\begin{rem} There exist 
non-rational stably rational varieties, see \cite{BCTSSD85}.
However, we do not know whether every stably 
uniformly rational variety is uniformly rational.
One may also ask whether there exists a non-rational
stably uniformly rational variety. 
\end{rem} 
On the other hand, we have the following 
lemma suggested by L\'arusson, 
see \cite[Proposition 1.9]{KKT18}. 
\begin{lem}\label{lem:Larusson}
If the product  $X=X_1\times X_2$ 
of smooth varieties $X_1$ and $X_2$
is elliptic, then the $X_i$ are elliptic. 
\end{lem}
For the reader's convenience we sketch the proof. 
\begin{proof}
let $(E,p,s)$ be a dominating 
spray on $X$. 
Pick a point $P\in X_2$ and consider 
the restriction 
$p_1\colon E_1\to X_1$ of $p\colon E\to X$ 
to $X_1\times\{P\}\subset X$.
Letting now 
\[s_1= {\rm pr}_1\circ s|_{E|_{X_1\times\{P\}}}
\colon E_1\to X_1\]
yields a desired dominating spray $(E_1,p_1,s_1)$ 
on $X_1$,
cf.\ the proof of Proposition 1.9 in \cite{KKT18}. 
\end{proof}
\begin{cor}\label{cor:lar} 
A complete stably uniformly rational variety $X$ is elliptic. 
\end{cor}
\begin{proof}
Let $X\times\A^k$ be uniformly rational. 
Then also $\hX=X\times\PP^k$ is. 
By Theorem \ref{thm:main}(i) $\hX$ is elliptic. 
By L\'arusson's Lemma \ref{lem:Larusson}
$X$ is elliptic too.
\end{proof}
The ampleness assumption in Theorem 
\ref{thm:main}(ii) is not a necessary one. Indeed, 
we have the following version of this theorem.
\begin{thm}\label{thm:fb}
Let $\pi \colon X\to B$ be a locally trivial fiber bundle 
with the base $B$ and the fiber $V$ being 
 uniformly rational smooth complete varieties.
Let $D$ be a relatively ample divisor 
on $X$, that is, $D\cdot X_b$ is an ample divisor 
on the fiber $X_b=\pi^{-1}(b)\simeq V$ for all $b\in B$. 
Let $F=\cO_X(\pm D)$.
Then $X$ and $Y=F\setminus Z_F$ are elliptic. 
\end{thm}
\begin{proof}
The ellipticity of $X$ follows from Lemma 
\ref{lem:ell} and the remark after Theorem \ref{uni.t1}. 
For every $b\in B$ the conclusion of 
Lemma \ref{lem:data} holds for 
the fiber $X_b$. It follows that $X_b$ verifies  
the strengthened curve-orbit property $(*)$, 
see Lemma \ref{cor:2-orbits}, that is, for any $x\in X_b$ 
there is a rational curve $C_x$ on $X_b$ smooth at $x$
 and a pair of  $\GG_{\mathrm a}$-like spays 
 $(E_i, p_i,s_i)$, $i=1,2$ on $X_b$ as in Remark 
 \ref{rem:without-restriction}. 
 Due to the local triviality of $\pi$ 
 and Gromov's Localization Lemma 
 these sprays can be extended on $X$. 
 Now the argument from the proof of 
 Proposition \ref{prop:ell-cone} 
 applies and yields the ellipticity of $Y$.
\end{proof} 
Observe that a relatively ample divisor $D$ 
as in Theorem \ref{thm:fb} 
does not need to be ample on $X$; cf.\ e.g., 
\cite[Example 2.11]{KZ23b}. 
\section{Examples} \label{sec:exs}
We start with the following well known example. 
For the reader's convenience we provide an argument.
\begin{lem}\label{lem:toric}
A smooth  complete toric variety $X$ belongs 
to class $\cA_0$.
\end{lem}
\begin{proof} Let $N$ be a lattice of rank $n=\dim(X)$ 
and $\Sigma$ be the complete fan in 
$N_\QQ=N\otimes_\ZZ\QQ$ associated with  $X$. 
Since $X$ is smooth every $n$-cone 
$\sigma\in\Sigma$ is simplicial generated by  vectors 
$v_1,\ldots,v_n\in N$ which form a base of $N$. 
The corresponding affine variety $X_\sigma$ 
is isomorphic to $\A^n$ 
and the affine charts  $X_\sigma$ cover $X$, 
see, e.g., \cite[Sec.~1.4]{Ful93}. 
\end{proof}
For the proof of the next lemma we address the original paper. 
\begin{lem}[{\rm \cite[Theorem 5]{APS14}}]\label{lem:toric-1}
A smooth  complete rational variety $X$ with a torus 
action of complexity $1$ belongs to class $\cA_0$.
\end{lem}
\begin{rem}
As a weak analog of the last result for affine varieties, let us mention
that every smooth rational affine variety 
with a torus action of complexity $0$ or $1$ is uniformly rational, 
see \cite{LP19}. 
See also \cite{Pet17} for examples of smooth contractible 
uniformly rational affine threefolds with a torus action of complexity 2
non-isomorphic to $\A^3$. 
These include, in particular, the famous Koras-Russell cubic.
\end{rem}
Lemma \ref{lem:toric} extends to 
complete spherical varieties. 
For the reader's convenience we sketch a proof.
\begin{lem}[{\rm\cite[Sec. 1.5, Corollaire]{BLV86}}]
\label{lem:spherical}
A smooth complete spherical variety $X$ 
belongs to class $\cA_0$.
\end{lem}
\begin{proof} 
Recall first the Local Structure Theorem, 
see~\cite[Th\'eor\`eme 1.4]{BLV86}. 
Consider a normal $G$-variety $Z$, 
where $G$ is a connected reductive algebraic group.
Let $z\in Z$ be a point  
such that the orbit $Gz$ is a projective variety.
Then the stabilizer $G_z$ is a parabolic subgroup.
Let $P\subset G$ be the opposite parabolic subgroup  
with the Levi decomposition $P=LP^u$ where
$P^u$ is the unipotent radical of $P$ and $L=P\cap G_z$ 
is the Levi subgroup in $P$. 
The Local Structure Theorem asserts 
that there is a locally closed affine subset $V\subset Z$ 
such that $z\in V$, $LV=V$, $P^uV$ is open in $Z$ 
and the action of $P^u$ on $Z$
defines an isomorphism $P^uV\cong P^u\times V$. 

Returning to Lemma~\ref{lem:spherical} we let $Z=X$ 
and we apply the notation above. 
It suffices to show that any point $x\in X$ 
whose orbit $Gx$ is closed in $X$ 
admits a neighborhood in $X$
isomorphic to $\mathbb{A}^n$. 
Indeed, for any $y\in X$ the closure of the orbit $Gy$ 
contains such a point $x$. 

Since $X$ is spherical, in the setup of the Local Structure Theorem 
the  reductive Levi subgroup $L$ 
of ${P=G_x^-}$ acts 
on the corresponding smooth affine variety $V$ 
with an open orbit 
and with an $L$-fixed point $x$.
Applying Luna's \'Etale Slice theorem,
see, e.g., \cite[Corollary of Theorem~6.7]{PV94}, 
we deduce that $V$ is equivariantly isomorphic 
to an $L$-module. 
It follows that  $x$ has a neighborhood in $X$ isomorphic to 
$P^u\times V$, which is an affine space. 
\end{proof}
\begin{rem}
A similar argument proves the following fact, 
see \cite[Theorem 3]{Pop20}. Let
$G$ be a connected reductive algebraic group 
and $X$ be a smooth affine $G$-variety. 
Assume that $\cO_X(X)^G =\kk$ and
the unique closed $G$-orbit $O$ in $X$ is
rational. Then $X$ is uniformly rational. 
\end{rem}
Due to Lemmas \ref{lem:toric}, 
\ref{lem:toric-1} and \ref{lem:spherical} 
the following corollary of Theorem \ref{thm:main}  
is immediate. 
\begin{cor} \label{cor:tor-sph}
The punctured affine and generalized affine cones
over a smooth projective 
spherical variety
equipped with an ample polarization 
are elliptic. This remains true after successively 
blowing up such a variety 
in smooth subvarieties.
The same holds for a
smooth projective rational variety $X$ 
with a torus action of complexity~$1$. 
\end{cor}
\begin{rem}\label{rem:flexibility}
Recall that the 
normalization of the affine cone 
${\rm cone}(X)$ over a smooth 
toric variety $X$ in $\PP^n$ is a normal 
affine toric variety with no torus factor.
It is known that such a variety is flexible, 
see \cite[Theorem 0.2(2)]{AKZ12}.
It follows that ${Y={\rm cone}(X)\setminus\{0\}}$ 
is elliptic. 
However, it is not clear whether the flexibility 
of ${\rm cone}(X)$ survives under 
blowing-up a point in $X$.
\end{rem}
Let us mention some known examples 
of uniformly rational smooth Fano varieties. 
\begin{exas} \label{exa:bb} 
1.
 It is known that a smooth rational cubic 
 hypersurface in $\PP^{n+1}$, $n\ge 2$
and a smooth intersection of 
two smooth quadric hypersurfaces in 
$\PP^{n+2}$, $n\ge 3$
are uniformly rational, cf.\ \cite[3.5.E$'''$]{Gro89} 
and \cite[Examples 2.4 and 2.5]{BB14}.
According to Theorem \ref{thm:main}, 
such a variety $X$
is  elliptic and the punctured (generalized) 
affine cone $Y$ over $X$ equipped with 
an ample polarization is elliptic. 

Notice that any smooth cubic hypersurface of dimension 
$n\ge 2$ in $\PP^{n+1}$  is unirational, see \cite[Theorem 1.38]{KSC04}. 
All smooth cubic surfaces in $\PP^3$ are rational. By contrast,  
no smooth cubic threefold in $\PP^4$ is rational, see \cite{CG72}. 
It is unknown (but is plausible) 
whether for $n\ge 4$ the general cubic hypersurface  in 
$\PP^{n+1}$ is irrational. 
However, for any $k\ge 1$ there are smooth cubic hypersurfaces 
in $\PP^{2k+1}$ which contain two disjoint linear $k$-subspaces. 
Any such cubic hypersurface is rational (hence also uniformly rational),
see \cite[1.33--1.35]{KSC04}. 

A smooth intersection $X$ of 
two smooth quadric hypersurfaces in 
$\PP^{n+2}$, $n\ge 3$
 is always rational, 
see, e.g., \cite[p. 796]{GH94}.
For $n=3$ no smooth quartic threefold
$X$ as above belongs to class $\cA_0$. 
The latter
follows from the classification of smooth 
Fano threefolds with Picard number 1 
which contain an open subset isomorphic to 
$\A^3$,
see \cite[Theorem 4.31]{CPPZ21}. 

2. It is shown in 
\cite[Proposition 3.2]{BB14} that 
the moduli space $\overline{\mathcal{M}}_{0,n}$ 
of stable $n$-pointed rational curves
is a complete uniformly rational  variety. 
By Theorem \ref{thm:main} it is elliptic, 
hence is the image of an affine space under 
a surjective morphism. 
 
3. The same holds for 
a small algebraic resolution of
a nodal cubic threefold in $\PP^4$, see 
\cite[Example 2.4 and Theorem 3.5]{BB14}. 
See also \cite[Section 3]{BB14} 
for further examples.

4. Any smooth Fano-Mukai fourfold with Picard number $1$,
genus $10$ and index $2$ is rational.
The moduli space $\mathcal{M}$ of such fourfolds 
is one-dimensional. 
With one exception, every such fourfold $X_t$ belongs 
to class $\cA_0$, see \cite[Theorem 2]{PZ23}. 
The exceptional 
Fano-Mukai fourfold $X_0$ is covered 
by open $\A^2$-cylinders
$Z_i\times\A^2$ where the $Z_i$ are smooth 
rational affine surfaces, 
see \cite[Proposition 7.3]{PZ23}. 
Since the $Z_i$ 
are uniformly rational, 
so is $X_0$.
By Theorem \ref{thm:main} the punctured 
affine cones 
and generalized affine cones over 
$X_t\in\mathcal{M}$
are elliptic. 
In fact, all these cones over $X_t$, 
without any exception, 
are even flexible, 
see \cite[Theorem 1]{PZ23}. 
\end{exas}
\medskip
\noindent {\bf Acknowledgments.} 
We are grateful to Yuri Prokhorov for valuable discussions, 
to draw our attention to
 uniformly rational varieties, to
the articles \cite{BB14, BHSV08} and to
Example \ref{exa:bb}.1. 
Our thanks are also due to an anonymous referee 
for useful remarks. 
%

\end{document}